\documentclass[12pt]{amsart}

\usepackage{pinlabel}
\usepackage{amsmath,amsfonts,amssymb,latexsym,mathrsfs,stmaryrd}
\usepackage{color}
\usepackage{graphicx}
\usepackage{graphics}
\usepackage{enumerate}
\usepackage{amscd} 
\usepackage{amsxtra}
\usepackage{epic,eepic}
\usepackage{hyperref} 
\usepackage{enumitem}
\usepackage{MnSymbol}
\usepackage{dsfont}

\newtheorem{theorem}{Theorem}[section]

\newtheorem{lemma}[theorem]{Lemma}

\theoremstyle{definition} 

\newtheorem{definition}[theorem]{Definition}

\newcommand{\qu}{/\kern-.7ex/}
\newcommand{\lqu}{\backslash \kern-.7ex \backslash}
\newcommand{\on}{\operatorname}

\title[Wall-crossing in K-theoretic LG theory]{Wall-Crossing in Genus Zero K-theoretic Landau-Ginzburg Theory}

\author{Hsian-Hua Tseng}
\address{Department of Mathematics\\ Ohio State University\\ 100 Math Tower, 231 West 18th Ave.\\Columbus\\ OH 43210\\ USA}
\email{hhtseng@math.ohio-state.edu}

\author{Fenglong You}
\address{Department of Mathematics\\ Ohio State University\\ 100 Math Tower, 231 West 18th Ave.\\Columbus\\ OH 43210\\ USA}
\email{you.111@osu.edu}


\begin{document}
\date{\today}

\begin{abstract} 
For a Fermat quasi-homogeneous polynomial $W$, we study a family of K-theoretic quantum invariants parametrized by a positive rational number $\epsilon$. We prove a wall-crossing formula by showing the generating functions lie on the Lagrangian cone of the permutation-equivariant K-theoretic FJRW theory of $W$.
\end{abstract}

\maketitle 

\tableofcontents

\section{Introduction}

Landau-Ginzburg/Calabi-Yau correspondence (LG/CY) arises from a variation of the GIT quotient in Witten's gauged linear sigma model (GLSM) \cite{Witten}. LG/CY correspondence describes a relationship between sigma models based on Calabi-Yau hypersurfaces in weighted projective spaces and the Landau-Ginzburg model of the defining equation of the Calabi-Yau. The mathematical A-model on the Calabi-Yau side is given by Gromov-Witten (GW) theory. On the other hand, the mathematical A-model on the Landau-Ginzburg side is called Fan-Jarvis-Ruan-Witten (FJRW) theory formulated by Fan-Jarvis-Ruan \cite{FJR13}.

LG/CY correspondence was proved by Chiodo-Iritani-Ruan in genus zero for Calabi-Yau Fermat polynomials \cite{CIR}. They proved a mirror theorem relating the FJRW theory of a Fermat polynomial to a hypergeometric series. The analogous mirror theorem for GW theory was proved by Givental \cite{Givental96} and Lian-Liu-Yau \cite{LLY}. The LG/CY correspondence follows from analytic continuation on the global K\"ahler moduli space.

Ross-Ruan \cite{RR} obtained a new geometric interpretation of the Landau-Ginzburg mirror theorem of Chiodo-Iritani-Ruan \cite{CIR} by proving a wall-crossing formula that relates the generating functions of GLSM introduced by Fan-Jarvis-Ruan \cite{FJR15}, where a mathematically rigorous definition of GLSM is provided. For hypersurface,  the GLSM is a one-dimensional family of cohomological field theories (CohFTs) parametrized by the nonzero rational numbers. The family of CohFTs which arise over the positive rational numbers corresponds to the geometric phase. Wall-crossing formulas in the geometric phase have been studied by Ciocan-Fontanine-Kim \cite{CFK13}, \cite{CFK14a} and \cite{CFK14b}. The family of CohFTs which arise over the negative rational numbers corresponds to the Landau-Ginzburg phase. Ross-Ruan \cite{RR} proved wall-crossing formulas analogous to Ciocan-Fontanine-Kim for a Fermat polynomial. 

The goal of this paper is to prove an analog of Ross-Ruan's result in K-theory. K-theoretic Gromov-Witten invariants was defined by A. Givental and Y.-P. Lee \cite{GL}, \cite{Lee} as holomorphic Euler characteristics of vector bundles on the moduli space of stable maps. A new enriched version of the theory called permutation-equivariant K-theoretic Gromov-Witten theory has been introduced by Givental \cite{Givental15}. The permutation-equivariant theory takes into account the $S_n$ action on the moduli spaces permuting the marked points. It fits better in the framework of mirror symmetry. Givental \cite{Givental15} proved that certain $q$-hypergeometric series associated to toric manifolds lie on the Lagrangian cone of the permuation-equivariant K-theoretic Gromov-Witten theory. A ``quantum Lefschetz'' type theorem for an arbitrary smooth projective variety and a K-theoretic mirror theorem for toric fibrations were proved by Tonita \cite{Tonita}. A permutation-equivariant K-theoretic analogy of Ciocan-Fontanine-Kim \cite{CFK14a} was proved by Tseng-You \cite{TY}.

The definition of K-theoretic FJRW theory and the K-theoretic wall-crossing formula in genus zero between FJRW and GW of hypersurfaces has been worked out by J\'er\'emy Gu\'er\'e \cite{Guere}. More precisely, Gu\'er\'e proved his result using mirror symmetry as in the original LG/CY correspondence of Chiodo-Iritani-Ruan \cite{CIR}. Our result, on the other hand, follows the approach of Ross-Ruan \cite{RR}.

we now give a more detailed overview of our result.

For each $\epsilon \in \mathbb Q$, we consider the CohFT arises from the negative rational number $-\epsilon$ in the Landau-Ginzburg phase. The moduli space $\mathcal R^d_{\vec k, \epsilon \vec l}$ parametrizes $\epsilon$-stable pairs $(C,L)$, where $C$ is a Hassett-stable rational orbifold curve with $m$ orbifold marked points $x_1,\ldots, x_m$ and $n$ smooth light points $y_1,\ldots,y_n$, and $L$ is an orbifold line bundle which satisfies 
\[
L^{\otimes d}\cong \omega_{log}(-\sum l_i y_i).
\] 
The vector $\vec k=(k_1,\ldots, k_m)$ records the multiplicities of the line bundles at the orbifold marked points.

For a Fermat polynomial $W$, the K-theoretic invariant is defined by
\[
\langle \phi_{k_1}\mathbb L^{j_1},\ldots, \phi_{k_m}\mathbb L^{j_m}| \phi_{l_1},\ldots,\phi_{l_n}\rangle_{m,n}^{W,\epsilon}=
d\cdot\chi(\mathcal R^d_{\vec k,\epsilon \vec l},\mathcal O^{vir}_{\vec k, \epsilon \vec l}\otimes (\otimes_{i}^m \mathbb L_i^{\otimes j_i}) ),
\] 
where $\mathcal O^{vir}_{\vec k, \epsilon \vec l}$ is called the virtual structure sheaf, $\phi_i$ are elements of the so called narrow state space $K_W^\prime$, $j_i$ are nonnegative integers and $\mathbb L_i$ are line bundles over $\mathcal R^d_{\vec k,\epsilon \vec l}$ corresponding to the pullback of the $i$-th  tautological line bundles over the Hassett moduli space \cite{Hassett} via forgetful maps. We write 
\[
\langle \phi_{k_1}\mathbb L^{j_1},\ldots, \phi_{k_m}\mathbb L^{j_m}| \phi_{l_1},\ldots,\phi_{l_n}\rangle_{m,n}^{W,\epsilon, S_m\times S_n}
\]
for the permutation-equivariant invariants.

Genus zero wall-crossing formulas are naturally stated via generating functions of permutation-equivariant K-theoretic invariants. Let ${\bf t}(q)$ be a Laurent polynomial in $q$ with coefficients in $K^\prime_W\otimes \Lambda$, where $\Lambda$ is a $\lambda$-algebra, and ${\bf u}:=\sum u^k \phi_k\in K^\prime_W$. For a positive rational number $\epsilon$, the big $\mathcal J^\epsilon$-function is defined as
\begin{align*}
\mathcal J^\epsilon(t,u,q):=& (1-q)\phi_0\sum\limits_{\substack{a_i\geq0, i\in\on{nar}\\ \sum a_i\leq \lceil \frac 1 \epsilon \rceil}}\prod_i\left(\frac{u^i\phi_i}{1-q}\right)^{a_i}\prod_{j=1}^N\prod\limits_{\substack{0\leq b <q_j+ \sum_i a_i\langle i q_j\rangle\\ \langle b\rangle=\langle q_j+\sum_i a_i i q_j\rangle}}(1-q^{b})\\
& +{\bf t}(1/q)+\sum_k \sum_{m,n}\phi^k\langle \frac{\phi_k}{1-q\mathbb L_1},{\bf t}(\mathbb L)^m|{\bf u}^n\rangle_{1+m,n}^{\epsilon,S_m\times S_n},
\end{align*}
Theorem \ref{main-theorem} shows that the $\mathcal J^\epsilon$-function lies on the Lagrangian cone of the permutation-equivariant K-theoretic FJRW theory of $W$.

\subsection{Acknowledgments}
We want to thank J\'er\'emy Gu\'er\'e, Dustin Ross and Yongbin Ruan for discussions. H.-H. T. is supported in part by NSF grant DMS-1506551 and a Simons Foundation Collaboration Grant. F. Y. is supported by the Presidential Fellowship  at the Ohio State University.

\section{K-theoretic Weighted FJRW Theory}
\subsection{Moduli Spaces}\label{sec-moduli-space}
\begin{definition}[\cite{RR} Definition 1.1]\label{d-epsilon-stable-curve}
For a positive rational number $\epsilon$, a $(d,\epsilon)$-stable rational curve is a rational connected orbifold curve $C$ with at worst nodal singularity, together with $m$ distinct orbifold marked points $x_1,\ldots,x_m$ and $n$ (not necessarily distinct) smooth light marked points $y_1,\ldots,y_n$ satisfies the following

\begin{itemize}
\item all nodes and orbifold marked points $x_i$ have cyclic isotropy $\mu_d$ and the orbifold structure is trivial away from nodes and orbifold markings;
\item $\on{mult}_z(\epsilon \sum[y_i])\leq 1$ at every points $z$ in $C$;
\item $\omega_{log}(\epsilon[y_i]):=\omega_C(\sum[x_i]+\epsilon[y_i])$ is ample.
\end{itemize}
We write 
\[
\overline{M}^d_{0,m+\epsilon n}
\] 
for the moduli space of $(d,\epsilon)$-stable curves.
\end{definition}

\begin{definition}[\cite{RR} Definition 1.2]
For 
\[
\vec{l}:=(l_1,\ldots,l_n)
\]
with $0\leq l_i\leq d-1$, an $\vec l$-twisted $d$-spin structure on a $(d,\epsilon)$-stable curve consists of an orbifold line bundle $L$ and an isomorphism
\[
L^{\otimes d}\overset{\kappa}{\longrightarrow} \omega_{log}(-\sum l_i[y_i])
\]
We write 
\[
\mathcal R^d_{m,\epsilon \vec l}
\]
for the moduli space of $(d,\epsilon)$-stable curves with $\vec l$-twisted $d$-spin structure.
\end{definition}

The restriction of $L$ to an orbifold point $x$ is a character of $\mu_d$, that is, multiplication by $e^{2\pi ik/d}$, for some $0\leq k <d$. The multiplicity of $L$ at $x$ is defined by 
\[
\on{mult}_x L:=k. 
\]

\begin{definition}
For 
\[
\vec k=(k_1,\ldots,k_m)
\]
where $k_i\in\{0,1,\ldots, d-1\}$, we write 
\[
\mathcal R^d_{\vec k,\epsilon \vec l}
\]
for the component of $\mathcal R^d_{m,\epsilon \vec l}$ indexed by the multiplicities of the line bundle at the orbifold points $x_i$: 
\[
\on{mult}_{x_i}L=k_i+1 (\on{mod} d).
\]
\end{definition}

The forgetful maps to Hassett moduli spaces \cite{Hassett}
\[
\theta: \mathcal R^d_{\vec k,\epsilon \vec l}\rightarrow \overline{M}_{0,m+\epsilon n}
\]
are defined by forgetting the line bundle $L$ and the orbifold structure on $C$. We define line bundles $\mathbb L_i$ on $\mathcal R^d_{\vec k,\epsilon \vec l}$ by the pullback of the cotangent line bundle at the $i$-th marked point via $\theta$.

We write 
\[
\mathcal R^{G,d}_{\vec k, \epsilon \vec l}
\]
for the graph spaces parametrizing the objects of $\mathcal R^{d}_{\vec k, \epsilon \vec l}$ with degree one maps $f:C\rightarrow \mathbb P^1$. Namely, there is an irreducible component $\hat{C}$ of $C$ such that the map restricts to an isomorphism on $\hat{C}$ and contracts the remaining components $C\setminus\hat{C}$. The third condition in Definition \ref{d-epsilon-stable-curve} changes to $\omega_{log}(\epsilon\sum[y_i])$ is ample on $\overline{C\setminus \hat{C}}$. We also use $\theta$ to denote the forgetful maps
\[
\theta: \mathcal R^{G,d}_{\vec k, \epsilon \vec l}\rightarrow \overline{M}_{0,m+\epsilon n}(\mathbb P^1,1),
\]
and define line bundles $\mathbb L_i$ on $\mathcal R^{G,d}_{\vec k,\epsilon \vec l}$ by the pullback of the cotangent line bundle at the $i$-th marked point via $\theta$.

\subsection{State Space}
let $W(x_1,\ldots,x_N)$ be a quasi-homogeneous polynomial for which there exist charges $(q_1,\ldots,q_N)\in \mathbb Q^N$ such that for any $\lambda \in \mathbb C^*$:
\[
W(\lambda^{q_1}x_1,\ldots,\lambda^{q_N}x_N)=\lambda W(x_1,\ldots,x_N).
\]
We define $d$ and $w_i$ for $i\in \{1,\ldots, N\}$ to be the unique positive integers such that $q_i=w_i/d$ for $i\in \{1,\ldots, N\}$ and $\on{gcd}(w_1,\ldots,w_N,d)=1$. We define $q=\sum q_j$. We also assume that $q_j$ are uniquely determined from $W$ and the affine variety defined by $W$ is singular only at the origin. Equivalently, the hypersurface 
\[
X_W=\{W=0\}\subset W\mathbb P^{N-1}
\]
defined by $W$ in weighted projective space is nonsingular.

We will only use a small part of the state space and in our situation it can be simplified to the following: We define the extended narrow state space associated to $W$ to be 
\[
K_W:=\mathbb Q^d
\]
and write $\{\phi_k\}_{k=0}^{d-1}$ for the basis of $K_W$. Consider the set
\[
\on{nar}:=\{k: \text{ for all } j, \langle q_j(k+1)\rangle\neq 0\},
\] 
The narrow sector is defined by restricting $K_W$ to the vectors indexed by $\on{nar}$:
\[
K_W^\prime:=\oplus_{k\in \on{nar}}\mathbb Q_{\phi_k}
\]
There is a perfect pairing on the narrow sector defined by
\[
(\phi_i,\phi_j)_W:=\delta_{i+j,d-2}.
\]
We write $\{\phi^i\}$ for the dual basis of $\{\phi_i\}$ under this pairing.

\subsection{Virtual Structure Sheaf}
For the rest of the paper, we assume that $W$ is Fermat, that is, $w_j|d$ for all $j$. As in \cite{RR}, we define the integers $s_{ij}$ and $l_{ij}$ by
\[
l_i=:s_{ij}\frac{d}{w_j}+l_{ij}, \text{ for some } 0\leq l_{ij}< \frac{d}{w_j}.
\]
We define
\[
L_j:=L^{w_j}\otimes \mathcal O(\sum\limits_i s_{ij}[y_i])
\]
By concavity \cite[Lemma 1.5]{RR}, $R^1\pi_*L_j$ is a vector bundle when $k_i\in \on{nar}$. We write
\[
W_{\vec k,\epsilon \vec l}:=\bigoplus (R^1\pi_*L_j)^\vee.
\]
The virtual structure sheaf is defined by
\[
\mathcal O^{vir}_{\vec k, \epsilon \vec l}:= e^K(W_{\vec k,\epsilon \vec l})\in K(\mathcal R^d_{\vec k,\epsilon \vec l})\otimes \mathbb Q
\]
where the K-theoretic Euler class of a bundle $V$ is defined by
\[
e^K(V):=\sum_k (-1)^k\bigwedge^k V^*.
\]
The virtual structure sheaf is the K-theoretic counterpart of the Witten class in the cohomological FJRW theory.

\subsection{Invariants}
The K-theoretic weighted FJRW invariants are defined as follows
\[
\langle \phi_{k_1}\mathbb L^{j_1},\ldots, \phi_{k_m}\mathbb L^{j_m}| \phi_{l_1},\ldots,\phi_{l_n}\rangle_{m,n}^{W,\epsilon}:=
d\cdot\chi(\mathcal R^d_{\vec k,\epsilon \vec l},\mathcal O^{vir}_{\vec k, \epsilon \vec l}\otimes (\otimes_{i}^m \mathbb L_i^{\otimes j_i}) ),
\]
where $j_i$ are nonnegative integers and $\mathbb L_i$ are tautological line bundles over $\mathcal R^d_{\vec k,\epsilon \vec l}$ corresponding to the $i$-th orbifold marked points defined in Section \ref{sec-moduli-space}.
The invariants are defined to vanish if any of the $k_i, l_i$ are not narrow or if the underlying moduli space does not exist.

Similarly, the permutation-equivariant version of the invariants 
\[
\langle \phi_{k_1}\mathbb L^{j_1},\ldots, \phi_{k_m}\mathbb L^{j_m}| \phi_{l_1},\ldots,\phi_{l_n}\rangle_{m,n}^{W,\epsilon, S_m\times S_n}:=
d\cdot\chi(\mathcal R^d_{\vec k,\epsilon \vec l}/(S_m\times S_n),\mathcal O^{vir}_{\vec k, \epsilon \vec l}\otimes (\otimes_{i}^m \mathbb L_i^{\otimes j_i}) )
\]
are defined by the K-theoretic push forward of $\mathcal O^{vir}_{\vec k, \epsilon \vec l}\otimes (\otimes_{i}^m \mathbb L_i^{\otimes j_i})$ along the projection
\[
\pi: \mathcal R^d_{\vec k,\epsilon \vec l}/(S_m\times S_n)\rightarrow [pt].
\]

Let $\Lambda$ be a $\lambda$-algebra, that is, an algebra over $\mathbb Q$ equipped with abstract Adams operations 
\[
\Psi^k:\Lambda\rightarrow \Lambda, \quad k=1,2,\ldots.
\]
Ring homomorphisms $\Psi^k$ satisfy 
\[
\Psi^r\Psi^s=\Psi^{rs}\quad \text{and}\quad \Psi^1=\on{id}.
\] 
We assume that $\Lambda$ includes the algebra of symmetric polynomials in a given number of variables and $\Lambda$ has a maximal ideal $\Lambda_+$ with the corresponding $\Lambda_+$-adic topology.

We use double bracket notation to denote generating series
\[
\llangle \phi_{k_1}\mathbb L^{j_1},\ldots,\phi_{k_l}\mathbb L^{j_l}\rrangle_l^{W,\epsilon}(t,u):=
\sum\limits_{m,n}\langle \phi_{k_1}\mathbb L^{j_1},\ldots,\phi_{k_l}\mathbb L^{j_l},{\bf t}(\mathbb L)^m|{\bf u}^n\rangle_{l+m,n}^{W,\epsilon,S_m\times S_n}
\]
where ${\bf t}(q)$ is a Laurent polynomial in $q$ with coefficients in $K^\prime_W\otimes \Lambda$,
\[
{\bf u}:=\sum u^k \phi_k\in K^\prime_W, \text{ and } {\bf t}(\mathbb L)^m:={\bf t}(\mathbb L_1),\ldots,{\bf t}(\mathbb L_m).
\]
\subsection{The $J^\epsilon$ Functions}
The permutation-equivariant K-theoretic big $J^\epsilon$-function is the following generating function
\begin{align*}
\mathcal J^\epsilon(t,u,q):=& (1-q)\phi_0\sum\limits_{\substack{a_i\geq0, i\in\on{nar}\\ \sum a_i\leq \lceil \frac 1 \epsilon \rceil}}\prod_i\left(\frac{u^i\phi_i}{1-q}\right)^{a_i}\prod_{j=1}^N\prod\limits_{\substack{0\leq b <q_j+ \sum_i a_i\langle i q_j\rangle\\ \langle b\rangle=\langle q_j+\sum_i a_i i q_j\rangle}}(1-q^{b})\\
& +{\bf t}(1/q)+\sum_k \phi^k\llangle \frac{\phi_k}{1-q\mathbb L_1}\rrangle_1^\epsilon,
\end{align*}
where the multiplication on $K_W$ is defined by
\[
\phi_i\cdot \phi_j:=\phi_{i+j \on{mod} d}.
\]

When $\epsilon >1$, we have the big $J$-function
\[
\mathcal J^\infty(t,1/q)=(1-1/q)\phi_0+{\bf t}(q)+\sum_k \phi^k\llangle \frac{\phi_k}{1-\mathbb L_1/q}\rrangle_1^\infty(t).
\]
We write $\mathcal L_{S_\infty}$ for the range of the big $J$-function in permutation-equivariant quantum K-theory of Landau-Ginzburg model.

We write $\mathcal K$ for the space of rational functions of $q$ with coefficients from $K^\prime_W\otimes\Lambda$. The space $\mathcal K$ is equipped with a symplectic form
\[
\Omega(f,g)=:=-[\on{Res}_{q=0}+\on{Res}_{q=\infty}](f(q^{-1}),g(q))_W\frac{\,d q}{q}.
\]
It can be decomposed into the direct sum
\[
\mathcal K=\mathcal K_+\oplus \mathcal K_-,
\]
where $\mathcal K_+$ is the subspace of Laurent polynomials in $q$ and $\mathcal K_-$ is the complementary subspace of rational functions of $q$ regular at $q=0$ and vanishing at $q=\infty$.

Now we state the main result of our paper:

\begin{theorem}\label{main-theorem}
For all $\epsilon>0$, $\mathcal J^\epsilon(t,u,1/q)$ is a $\mathcal K[[u]]$-valued point of $\mathcal L_{S_\infty}$, in other words, $\mathcal J^\epsilon(t,u,1/q)$ is a formal series of the form
\[
(1-1/q)\phi_0+\hat{\bf t}(q)+\sum_k \phi^k\llangle \frac{\phi_k}{1-\mathbb L_1/q}\rrangle_1^\infty(\hat{t}),
\]
for some $\hat{\bf t}(q)={\bf t}(q)+O(u)\in \mathcal K_+[[u]]$.
\end{theorem}

\section{Proof of Theorem \ref{main-theorem}}

The proof of Theorem \ref{main-theorem} will follow from Lemma \ref{lemma:J-function-relation}, Lemma \ref{lemma:characterization} and the fact that the $\mathcal J^\epsilon$-function satisfies properties {\bf (1)}, {\bf (2)}, and {\bf (3)} of  Lemma \ref{lemma:characterization}, by the definition of the $\mathcal J^\epsilon$-function. The technique we use in the proof of  Lemma \ref{lemma:J-function-relation} and Lemma \ref{lemma:characterization} is called K-theoretic virtual localization (cf., for example, \cite[Theorem 3.3]{Qu}).

\begin{lemma}\label{lemma:J-function-relation}
For every $\epsilon >0$, the series
\begin{equation}\label{J-function-relation}
(\partial_{u^r}\mathcal J^\epsilon(t,u,q),\partial_{t_0^s}\mathcal J^\epsilon(t,u,1/q))
\end{equation}
has no pole at roots of unity in $q$ for all narrow $r, s$.
\end{lemma}
\begin{proof}
Consider the $\mathbb C^*$-action on $\mathbb P^1$
\[
\lambda[z_0,z_1]:=[\lambda z_0,z_1], \, \lambda \in \mathbb C^*.
\]
We consider elements $p_0,p_\infty\in K^0_{\mathbb C^*}(\mathbb P^1)$ defined by the restriction to the fixed points:
\[
p_0|_0=q, p_0|_\infty=1, \quad \text{and}\quad p_\infty|_0=1, p_\infty|_\infty=1/q.
\]
The $\mathbb C^*$ action on $\mathbb P^1$ naturally induces an action on $\mathcal R^{G,d}_{\vec k, \epsilon \vec l}$ and $\mathcal O^{G, vir}_{\vec k,\epsilon \vec l}$. Consider the equivariant series
\begin{align}\label{graph-space-series}
& \frac 1 d \sum\langle {\bf t}(L)^m,\phi_s|{\bf u}^n,\phi_r|\on{ev}_{m+1}^*(p_\infty)\otimes \on{ev}_{m+n+2}^*(p_0)\rangle_{m+1,n+1}^{G,\epsilon}:=\\
\notag& \sum \chi(\mathcal R_{(\vec k, s),\epsilon(\vec l,r)}^{G,d},\mathcal O^{vir}_{(m+1),\epsilon(n+1)}\otimes{\bf t}(1/L)^m\otimes {\on u}^n\otimes\on{ev}_{m+1}^*(p_\infty)\otimes \on{ev}_{m+n+2}^*(p_0))
\end{align}
where $\on{ev}_i$ are equivariant evaluation maps
\[
\on{ev}_i:\mathcal R_{(\vec k, s),\epsilon(\vec l,r)}^{G,d}\rightarrow \mathbb P^1.
\]
By definition, the equivariant series (\ref{graph-space-series}) has no pole at roots of unity in $q$. Therefore it remains to prove the claim: (\ref{J-function-relation})=(\ref{graph-space-series}). This is done via $\mathbb C^*$-localization calculation.

In the fixed loci, marked points and nodes of the curves are mapped to $0$ and $\infty$ via $f$. Invariants vanish when restricted to loci where $f(x_{m+1})=0$ or $f(y_{n+1})=\infty$. We denote a fixed loci by
\[
\iota: F_{\vec k_0,\vec l_0}^{\vec k_\infty, \vec l_\infty}\hookrightarrow \mathcal R_{(\vec k, s),\epsilon(\vec l,r)}^{G,d},
\]
where $\vec k_0,\vec k_\infty$ is a splitting of the vector $\vec k$ into subvectors of lengths $m_0,m_\infty$ over $0$ and $\infty$; $\vec l_0,\vec l_\infty$ is a splitting of the vector $\vec l$ into subvectors of lengths $n_0,n_\infty$ over $0$ and $\infty$.
By localization formula, the equivariant series (\ref{graph-space-series}) equals
\begin{equation}\label{localization-formula}
\sum_F \chi\left(F,\frac{\iota^*(\mathcal O^{vir}_{(\vec k, s),\epsilon(\vec l,r)}\otimes{\bf t}(1/L)^{\vec k}\otimes {\bf u}^{\vec l}\otimes\on{ev}_{m+1}^*(p_\infty)\otimes \on{ev}_{m+n+2}^*(p_0))}{e^K_{\mathbb C^*}(N_F)}\right)
\end{equation}
where the T-equivariant K-theoretic Euler class of a bundle $V$ is defined by
\[
e^K_T(V):=\on{tr}_{\lambda\in T}\left(\sum_k (-1)^k\Lambda^k V^*\right)
\]
and 
\[
{\bf t}(1/L)^{\vec k}={\bf t}(1/L_1)_{k_1}\otimes\cdots\otimes {\bf t}(1/L_m)_{k_m}
\]
with ${\bf t}(q)_k$ the coefficient of $\phi_k$ in ${\bf t}(q)$, similar for the notation ${\bf u}^{\vec l}$. 

A fixed locus is called stable if it has a node over both $0$ and $\infty$. For stable fixed loci, we have
\[
F_{\vec k_0,\vec l}^{\vec k_\infty,\vec l_\infty}\cong \mathcal R^d_{(\vec k_0,k),\epsilon(\vec l_0,r)}\times \mathcal R^d_{(\vec k_\infty,s,d-2-k),\epsilon \vec l_\infty}
\]
where $k$ is uniquely determined.
Following the analysis of stable terms in \cite[Lemma 2.1]{RR}, we have
\[
\iota^*(W_{(\vec k, s),\epsilon(\vec l,r)}^{G})\cong W_{(\vec k_0,k),\epsilon(\vec l_0,r)}\oplus W_{(\vec k_\infty,s,d-2-k),\epsilon \vec l_\infty} 
\]
Hence
\[
\iota^* e^K(W_{(\vec k, s),\epsilon(\vec l,r)}^{G})=e^K(W_{(\vec k_0,k),\epsilon(\vec l_0,r)})\otimes e^K(W_{(\vec k_\infty,s,d-2-k),\epsilon \vec l_\infty}).
\]
Then we compute the contribution of $\mathbb C^*$-equivariant K-theoretic Euler class of the normal bundle $N_F$. The contribution from smoothing the node at $0$ is
\[
1-(q/L)^{1/d}
\]
In addition there are $d-1$ contributions from ghost automorphisms corresponding to the node at $0$, they are
\[
1-\zeta^k(q/L)^{1/d}, \quad 1\leq k \leq d-1,
\]
for a primitive $d$-th root of unity.
Recall that these contributions are in the denominator of (\ref{localization-formula}), using the identity
\[
\sum_{k=1}^{d-1}\frac{1}{1-\zeta^k x}+\frac{1}{1-x}=\frac{d}{1-x^d}
\]
and adding these contribution up, we have contribution
\[
\frac{1-q/L}{d}
\]
in the denominator of (\ref{localization-formula}).
Similarly, the contribution from smoothing the node at $\infty$ and the contributions from ghost automorphisms add up to
\[
\frac{1-1/(qL)}{d}.
\]
The contribution from deforming the map to $\mathbb P^1$ is 
\[
(1-q)(1-1/q).
\]
Combining everything together, for stable fixed loci, we have the contribution of stable terms to (\ref{localization-formula}) is equal to
\begin{equation}\label{stable-term}
\langle {\bf t}(1/L)^{\vec k_0},\frac{\phi_k}{1-qL_{m_0+1}}|{\bf u}^{\vec l_0},\phi_r\rangle^\epsilon_{m_0+1,n_0+1}\cdot
\langle {\bf t}(1/L)^{\vec k_\infty},\phi_s,\frac{\phi_{d-2-k}}{1-1/(qL_{m_\infty+2})}|{\bf u}^{\vec l_\infty}\rangle^\epsilon_{m_\infty+2,n_\infty}.
\end{equation}
The first factor of (\ref{stable-term}) corresponds to the coefficient of $\phi^k$ in the stable terms of $\partial_{u^r}\mathcal J^\epsilon(t,u,q)$ and the second factor corresponds to the coefficient of $\phi_k$ in the stable terms of $\partial_{t_0^s}\mathcal J^\epsilon(t,u,1/q)$.

Unstable terms appear in the following two cases. The first case is when
\[
m_\infty=n_\infty=0.
\]
Its contribution to (\ref{localization-formula}) is equal to
\begin{equation}\label{unstable-term-1}
\langle {\bf t}(1/L)^{\vec k_0},\frac{\phi_s}{1-qL_{m_0+1}}|{\bf u}^{\vec l_0},\phi_r\rangle^\epsilon_{m_0+1,n_0+1}\cdot 1,
\end{equation}
where the second factor $1$ corresponds to the coefficient of $\phi_s$ in the unstable terms of $\partial_{t_0^s}\mathcal J^\epsilon(t,u,1/q)$.

The second case is when 
\[
m_0=0 \text{ and } n_0+1\leq 1/\epsilon.
\]
Following the analysis of unstable terms in \cite[Lemma 2.1]{RR}, we have
\[
\iota^* W^G\cong W_\infty\oplus\left(\bigoplus_j (R^1 \pi_* L_j|_{\hat{C}})^\vee\right),
\]
and \v Cech representatives of cohomology $H^1((\hat{C},L_j|_{\hat C})$ are
\[
\left\langle \frac{1}{x_0^{db}x_1^{q_j+\langle q_j r\rangle+\sum_i \langle q_j(\vec l_0)_i\rangle-b}}\left|  0<b <q_j+\langle q_j r\rangle+\sum_i \langle q_j(\vec l_0)_i\rangle, \langle b\rangle=\langle q_j+ q_j r+\sum_i q_j(\vec l_0)_i\rangle \right.\right\rangle
\]
where $x_0,x_1$ are the orbifold coordinates on $\hat C$, related to the coarse coordinates on $\mathbb P^1$ by
\[
z_0=x_0^d, z_1=x_1.
\]
Hence, we have
\[
e^K(\bigoplus_j (R^1 \pi_* L_j|_{\hat{C}})^\vee)=\prod_{j=1}^N \prod_{\substack{0<b <q_j+\langle q_j r\rangle+\sum_i \langle q_j(\vec l_0)_i\rangle\\ \langle b\rangle=\langle q_j+ q_j r+\sum_i q_j(\vec l_0)_i\rangle}}(1-q^b).
\]
The contribution from deforming the $n_0+1$ smooth marked points is 
\[
(1-q)^{n_0+1}.
\]
Hence, the contribution from the second case of the unstable loci is
\begin{equation}\label{unstable-term-2}
\left(\frac{u^{\vec l_0}}{(1-q)^{n_0}}\prod_{j=1}^N \prod_{\substack{0<b <q_j+\langle q_j r\rangle+\sum_i \langle q_j(\vec l_0)_i\rangle\\ \langle b\rangle=\langle q_j+ q_j r+\sum_i q_j(\vec l_0)_i\rangle}}(1-q^b)\right)\cdot \langle {\bf t}(1/L)^{\vec k_\infty},\phi_s,\frac{\phi_{k}}{1-1/(qL_{m_\infty+2})}|{\bf u}^{\vec l_\infty}\rangle^\epsilon_{m_\infty+2,n_\infty},
\end{equation}
where
\[
k= r+\sum(\vec l_0)_i \mod d.
\]
The factor 
\[
\left(\frac{u^{\vec l_0}}{(1-q)^{n_0}}\prod_{j=1}^N \prod_{\substack{0<b <q_j+\langle q_j r\rangle+\sum_i \langle q_j(\vec l_0)_i\rangle\\ \langle b\rangle=\langle q_j+ q_j r+\sum_i q_j(\vec l_0)_i\rangle}}(1-q^b)\right)
\]
corresponds to the coefficient of $\phi_k$ in the unstable terms of $\partial_{u^r}\mathcal J^\epsilon(t,u,z)$.

Adding all the contributions from stable terms (\ref{stable-term}) and unstable terms (\ref{unstable-term-1}), (\ref{unstable-term-2}) together proves (\ref{J-function-relation})=(\ref{graph-space-series}), hence proves the lemma.
\end{proof}

\begin{lemma}\label{lemma:characterization}
Suppose $F(t,u,q)\in \mathcal K[[u]]$ has the form
\[
F(t,u,q)=(1-q)\phi_0+{\bf t}(1/q)+f(u,1/q)+\bar{F}(t,u,q)
\]
satisfies
\begin{description} 
\item[(1)] $f(u,q)$ is a Laurent polynomial in $q$ with coefficient in $K_W^\prime[[u]]$ and satisfies $f(0,q)=0$,

\item[(2)] $\bar{F}(t,u,q)\in \mathcal K_-[[u]]$ only has terms of degree at least $2$ in $t,u$ and is of the form
\[
\sum f_{\xi,\vec m,\vec n,j,s}\frac{t^{\vec m}u^{\vec n}(\xi q)^j}{(1-\xi q)^{j+1}}\phi^s,
\]
where $\xi$ is a root of unity, $t^{\vec m}=\sum (t^k_j)^{m^k_j}$ and similar for $u^{\vec n}$.
\item[(3)] $F(t,0,1/q)\in \mathcal L_{S_\infty}$.
\end{description}
Then $F(t,u,1/q)\in \mathcal L_{S_\infty}$ if and only if the series
\begin{equation}\label{partial-der-series}
(\partial_{u^r}F(t,u,q),\partial_{t_0^s}F(t,u,1/q))
\end{equation}
has no pole at roots of unity in $q$ for all narrow $r, s$. 
\end{lemma}
\begin{proof}
Suppose $F(t,u,q)$ lies on $\mathcal L_{S_\infty}$ and satisfies properties (1), (2) and (3). Therefore, $F$ has the form
\[
F(t,u,q)=(1-q)\phi_0+\hat{\bf t}(1/q)+\sum_k \phi^k\llangle \frac{\phi_k}{1-qL}\rrangle_1^\infty(\hat {\bf t}),
\]
where 
\[
\hat{\bf t}(q)={\bf t}(q)+f(u,q).
\] 
Following the same localization process as Lemma \ref{lemma:J-function-relation}, we have 
\[
\frac {1}{d}\llangle \partial_{u^r}\hat{\bf t}(L),\phi_s|\on{ev}^*_1(p_0)\otimes \on{ev}^*_2(p_\infty)\rrangle_2^{G,\infty}=(\partial_{u^r}F(t,u,q),\partial_{t_0^s}F(t,u,1/q))
\]
has no pole at roots of unity in $q$ for all narrow $r,s$.

Now suppose $F$ satisfies properties (1), (2), (3) and the series (\ref{partial-der-series})
has no pole at roots of unity for all narrow $r,s$. We want to show that $F\in \mathcal L_{S_\infty}$. Fix a root of unity $\xi$, assume we know $f_{\xi,\vec m^\prime,\vec n^\prime,j^\prime,s^\prime}$ for all $(|\vec n^\prime|,|\vec m^\prime|)\leq (|\vec n|,|\vec m|)$. The coefficient of $\frac{t^{\vec m}u^{\vec n}(\xi q)^j}{(1-\xi q)^{j+1}}$ in
$(\partial_{u^r}F(t,u,q),\partial_{t_0^s}F(t,u,1/q))$ is
the sum of the leading term $(n^r+1)f_{\xi,\vec m,(\vec n,r),j,s}$ and terms that are determined by induction and $f(u,q)$. This summation is zero since (\ref{partial-der-series}) has no pole at roots of unity in $q$. Therefore, $f_{\xi,\vec m,(\vec n,r),j,s}$ is recursively determined. This implies that $F(t,u,q)$ is recursively determined from $F(t,0,q)$. Therefore $F(t,u,q)$ and the series
\[
(1-q)\phi_0+\hat{\bf t}(1/q)+\sum_k \phi^k \llangle \frac{\phi_k}{1-qL}\rrangle_1^\infty(\hat{\bf t})
\]
agree on the restriction $u=0$ and satisfy the same recursion relation. Hence
\[
F(t,u,q)=(1-q)\phi_0+\hat{\bf t}(1/q)+\sum_k \phi^k \llangle \frac{\phi_k}{1-qL}\rrangle_1^\infty(\hat{\bf t})
\]
lies on $\mathcal L_{S_\infty}$.
\end{proof}

\end{document}